\documentclass[12pt]{article} 
\usepackage[utf8]{inputenc}

\usepackage{amsmath, amssymb, amsthm, amssymb}
\usepackage{mathdots}
\usepackage[pdftex]{graphicx}
\usepackage{fancyhdr}
\usepackage[margin=1in]{geometry}
\usepackage{multicol}
\usepackage{bm}
\usepackage{listings}
\PassOptionsToPackage{usenames,dvipsnames}{color}  
\usepackage{pdfpages}
\usepackage{algpseudocode}
\usepackage{tikz}
\usepackage{enumitem}
\usepackage[T1]{fontenc}
\usepackage{inconsolata}
\usepackage{framed}
\usepackage{wasysym}
\usepackage[thinlines]{easytable}
\usepackage{wrapfig}
\usepackage{hyperref}
\usepackage{mathrsfs}
\usepackage{tabularx} 
\usepackage{diagbox}
\usepackage{float}

\title{Limit Theorems for Descents in Permutations and Arithmetic Progressions in $\Z/p\Z$}
\author{Bryce Cai, Annie Chen, Ben Heller, Eyob Tsegaye}
\date{August 29, 2018}

\definecolor{shadecolor}{gray}{0.9}

\newcommand{\f}[2]{\frac{#1}{#2}}
\newcommand{\p}[1]{\left(#1\right)}
\newcommand{\s}[1]{\left[#1\right]}
\newcommand{\abs}[1]{\left\lvert#1\right\rvert}

\newcommand{\R}{\mathbb{R}}
\newcommand{\Var}{\mathrm{Var}}
\newcommand{\Wass}{\mathrm{Wass}}
\newcommand{\Kolm}{\mathrm{Kolm}}

\renewcommand{\P}{\mathbb{P}}

\newcommand{\E}{\mathbb{E}}

\newcommand{\bbn}{\mathbb{N}}
\newcommand{\bbr}{\mathbb{R}}

\newcommand{\Z}{\mathbb{Z}}

\newcommand{\bfx}{\mathbf{x}}

\newcommand{\norm}[1]{\left\lVert#1\right\rVert}
\renewcommand{\d}{\mathrm{d}}
\newcommand{\eps}{\varepsilon}

\newtheorem{thm}{Theorem}[section]
\newtheorem{prop}[thm]{Proposition}

\newtheorem{lem}[thm]{Lemma}
\newtheorem*{prop*}{Proposition}
\newtheorem*{cor*}{Corollary}
\newtheorem*{lem*}{Lemma}
\newtheorem*{thm*}{Theorem}

\theoremstyle{definition}
\newtheorem{conj}[thm]{Conjecture}
\newtheorem{defn}[thm]{Definition}

\newtheorem*{conj*}{Conjecture}

\theoremstyle{definition}
\newtheorem*{defn*}{Definition}
\newtheorem{remark}[thm]{Remark}
\newtheorem*{remark*}{Remark}

\begin{document}

\maketitle

\begin{abstract}
We prove a quantitative local limit theorem for the number of descents in a random permutation. Our proof uses a conditioning argument and is based on bounding the characteristic function $\phi(t)$ of the number of descents.

We also establish a central limit theorem for the number of 3-term arithmetic progressions (3-APs) in a random subset of $\Z/p\Z$.
We conjecture that there is no local limit theorem for 3-APs, but a proof of this remains elusive. A promising avenue of proof is to condition on the size of the subset and show that the resulting distributions are too far apart for different sizes. This has proven difficult because the distances between these conditioned distributions on are the same order as their standard deviations such that the constant multiple between them is not very large.
\end{abstract}

\section{Introduction}

Let $\{X_n\}$ be a sequence of discrete random variables taking integer values with mean $\mu_n$ and standard deviation $\sigma_n$. For example, $X_n$ could be the number of heads that occur in $n$ coinflips. Our goal is to understand to what degree such a sequence converges to a normal distribution. In the example of coinflips, the $X_n$ themselves do not converge in any meaningful way, since $\mu_n$ is increasing in $n$, so we normalize by considering convergence of $Y_n = \f{X_n - \mu_n}{\sigma_n}$ to the standard normal distribution $Z$ instead.

One notion of convergence of a sequence of random variables is convergence in distribution. This concerns pointwise convergence of cumulative distribution functions. We write that $\{Y_n\}$ converges in distribution to $Z$, or $Y_n \xrightarrow{d} Z$, if
\[ \abs{\P(Y_n \leq t) - \P(Z \leq t)} \to 0 \]
for each $t \in \R$. If this condition holds, the sequence $\{X_n\}$ satisfies a \textbf{central limit theorem}. The example of coinflips satisfies a central limit theorem. But what if we wanted to know, for instance, the probability of getting exactly half heads, that is, $\P(X_n = n/2)$ (for even $n$)? The central limit theorem only tells us the probability of having at most half heads, that is, $\P(X_n < \f{n}{2}) = \P(Y_n < 0) \to \P(Z < 0) = \f{1}{2}$. To satisfy this, we would need
\[ \abs{\P(X_n = k) - \f{1}{\sqrt{2\pi}\sigma_n} e^{-\p{\f{k - \mu_n}{\sigma_n}}^2/2}} \to 0, \]
but this is not meaningful since both terms go to $0$ trivially. To find an error bound that makes the convergence meaningful, observe that a normal distribution with standard deviation $\sigma$ has height $\Theta(\f{1}{\sigma})$. Therefore, we want
\[ \abs{\P(X_n = k) - \f{1}{\sqrt{2\pi}\sigma_n} e^{-\p{\f{k - \mu_n}{\sigma_n}}^2/2}} = o\p{\f{1}{\sigma_n}} \]
uniformly in $k$. If this condition holds, the sequence $\{X_n\}$ satisfies a \textbf{local limit theorem}. Equivalently, we require
\[ \abs{\sigma_n \P\p{\f{X_n-\mu_n}{\sigma_n} = k} - \f{1}{\sqrt{2\pi}} e^{-k^2/2}} = o(1). \]
Geometrically, this is equivalent since when the Gaussian is normalized, its height is scaled by $\sigma_n$. If there is a bound on this last rate of convergence, the local limit theorem is said to be \textit{quantitative}. Equivalently, there is a quantitative bound if the rate of convergence of the first difference is better than $o(1/\sigma_n)$.

\subsection{Descents} \label{descents-intro}

One variable of interest to us is the number of descents in a random permutation. A permutation $\pi$ has a descent at index $j$ if $\pi(j) > \pi(j+1)$. The number of descents in $\pi$, denoted $D(\pi)$, is the number of such indices $j$. We define the indicator random variable for the $j$th descent
\[ X_j = \begin{cases} 1 & \pi(j) > \pi(j+1) \\ 0 & \text{else} \end{cases}, \]
so the number of descents can be written
\[ D_n = \sum_{j=1}^{n-1} X_j. \]
To sample random permutations, uniformly and independently pick random integers $1 \leq a_j \leq n - j + 1$ for each $1 \leq j \leq n$. The $a_j$ correspond to permutations like so: let $S = \{1, \ldots, n\}$, and for each $j$, in order from $1$ to $n$, define $\pi(j)$ to be the $j$th remaining element of $S$, and remove $\pi(j)$ from $S$. Under this correspondence, $\pi(j) > \pi(j+1)$ if and only if $a_j > a_{j+1}$. Therefore, $X_j$ can be equivalently defined by
\[ X_j = \begin{cases} 1 & a_j > a_{j+1} \\ 0 & \text{else} \end{cases}. \]
	With this method of sampling, we can compute some basic facts about $D$. At each index $j$, there is either a descent or an ascent ($\pi(j) < \pi(j+1)$) and these both occur with equal probability, so $\E[X_j] = \f{1}{2}$. Therefore, $\E[D] = \f{n-1}{2}$. With similar reasoning as in computing the expectation of $X_j$, we get $\E[X_j X_{j+1}] = \f{1}{3!} = \f{1}{6}$. Also, it is important to note that $X_j$ is independent of all other $X_k$ except for $X_{j-1}$ and $X_{j+1}$. Now we compute
\begin{align*}
	\Var(D)
	&= \E[D^2] - \E[D]^2 \\
	&= \E\s{\sum_{i=1}^{n-1} \sum_{j=1}^{n-1} X_i X_j} - \p{\f{n-1}{2}}^2 \\
	&= \sum_{i=1}^{n-1} \E[X_i^2] + \sum_{\abs{i-j} = 1} \E[X_i X_j] + \sum_{\abs{i-j} > 1} \E[X_i X_j] - \f{(n-1)^2}{4} \\
	&= \f{n-1}{2} + \f{2(n-2)}{6} + \f{(n-1)^2 - (n-1) - 2(n-2)}{4} - \f{(n-1)^2}{4} \\
	&= \f{n+1}{12}.
\end{align*}

In Section 3, we establish a quantitative local limit theorem for $D_n$:
\begin{thm}
\label{thm:llt-dn}
\[ \abs{P\p{D_n = x} - \f{1}{\sqrt{2\pi}\sigma_n} e^{-\p{\f{x - \mu_n}{\sigma_n}}^2/2}} = O\p{n^{-1 + \eps}} = O\p{\f{1}{\sigma_n}n^{-\f{1}{2}+\eps}}. \]
\end{thm}

\subsection{3-term arithmetic progressions}

The other random variable of interest is the number of 3-term arithmetic progressions (3-APs) in a random subset of $\Z/n\Z$, with the presence of each element determined by a coinflip. We require $n$ prime to avoid issues of divisibility. For $1 \leq i \leq n$, let $x_i$ be the indicator random variable for $i \in S \subseteq \Z/n\Z$. The $x_i$ are independent and each has expectation $\f{1}{2}$. We define the number of 3-term arithmetic progressions in $S$ as
\[ A_n = \f{1}{2} \sum_{i=1}^{n} \sum_{j=1}^{n-1} x_i x_{i+j} x_{i+2j}. \]
This definition does not consider triples that contain the same element twice to be an arithmetic progression, and the factor of $\f{1}{2}$ causes multiple triples with the same elements to be considered the same progression. The expectation of $A_n$ is $\E[A_n] = \f{1}{8}\binom{n}{2}$. The variance of $A_n$ is computed in section 4.

In section 4, we prove a central limit theorem for $A_n$, stated as follows:
\begin{thm}
\label{thm:clt-an}
$$\abs{P\p{\frac{A_n - \mu_n}{\sigma_n} \leq x} - \P(Z \leq x)} = O\p{n^{-1/4}}.$$
\end{thm}

Based on experimental results and heuristics in section 4, we conjecture that there is no local limit theorem for $A_n$.
\begin{conj}
\label{conj:llt-an}
$$\abs{P\p{A_n = x} - \f{1}{\sqrt{2\pi}\sigma_n} e^{-\p{\f{x - \mu_n}{\sigma_n}}^2/2}} \neq o\p{\frac{1}{\sigma_n}}.$$
\end{conj}


\section{Background}

\subsection{Number of triangles in a random graph}
Previous papers have studied the distribution of the number of triangles $T_n$ in a random graph $G_{n,p}$, which is the undirected graph on $n$ vertices where each of the $\binom{n}{2}$ edges have probability $p$ of appearing in the graph. Let $R_n = \frac{T_n-\mu}{\sigma}$ be the normalized $T_n$ and $\mathcal{N}(x) = \frac{1}{\sqrt{2\pi}}e^{-x^2/2}$ be the standard normal probability density function. Gilmer and Kopparty \cite{GilmerKopparty14} establish the existence of an LLT for $T_n$, so $T_n$ is pointwise approximated by a discrete Gaussian distribution. 
\begin{thm}{(Gilmer, Kopparty; 2014)}
Uniformly for all $k \in \Z$,
\[ \P(T_n = k) = \f{1}{\sqrt{2\pi}\sigma_n} \mathrm{exp}\p{-((k-\mu_n)/\sigma_n)^2/2} + o(1/\sigma_n). \]
\end{thm}

We note that this is a qualitative result, where the error is $o(n^{-2}) = o(\frac{1}{\sigma})$. Berkowitz \cite{Berkowitz17} expanded on their work and established a quantitative bound on the distance between the distribution of $T_n$ and the Gaussian distribution. In particular, he shows the following result:

\begin{thm}{(Berkowitz; 2017})
Uniformly for all $k \in \Z$,
\[ \P(T_n = k) = \f{1}{\sqrt{2\pi}\sigma_n} \mathrm{exp}\p{-((k-\mu_n)/\sigma_n)^2/2} + O(n^{-2.5+\epsilon}). \]
\end{thm}

\subsubsection{Methods used to establish an LLT for $T_n$}
Here we summarize the methods used by Gilmer and Kopparty \cite{GilmerKopparty14} in proving that an LLT exists for $T_n$.
With some calculations, we have the mean $\E[T_n] = \mu_n = p^3{\binom{n}{3}}$ and the variance $\sigma^2 = \Theta(n^4).$ 
A crucial formula for the proof is the following:
\begin{prop}[Fourier Inversion Formula]
If $Y$ is a random variable with support in the discrete lattice $L = \frac{1}{b}(\Z - a)$ for $a,b \in \bbr$, and $\phi(t) = \E[e^{itY}]$ is the characteristic function of $Y$, then for all $y \in L$, 
\[ \P(Y = y) = \frac{1}{2\pi b} \int_{-\pi b}^{\pi b} e^{-ity}\phi(t)\ \d t.
\]
\end{prop}

If we let $\phi_n(t) = \E[e^{itR_n}]$, then 
$\sigma_n \P(R_n = x) = \frac{1}{2\pi} \int_{-\pi\sigma}^{\pi\sigma} e^{-itx}\phi(t)\ \d t.$
By the standard Fourier inversion formula,
$\mathcal{N}(x) = \frac{1}{2\pi} \int_{-\infty}^{\infty} e^{-itx}e^{x^2/2}\ \d t.$ Thus, 
\[ \abs{\sigma_n \P(R_n = x)- \mathcal{N}(x)} \leq \int_{-\pi\sigma_n}^{\pi\sigma_n} \abs{\phi_n(t) - e^{-t^2/2}} \d t + 2 \int_{\pi\sigma_n}^{\infty} e^{-t^2/2}\ \d t.
\]
Thus, since the second term goes to 0 as n goes to infinity, we just have to show \[\int_{-\pi\sigma_n}^{\pi\sigma_n} \abs{\phi_n(t) - e^{-t^2/2}} \d t = o(1).\]

For any constant $A$, we can write
	\[
		\int_{-\pi \sigma}^{\pi \sigma} \abs{\phi_n(t) - e^{-t^2/2}} \d t
		\leq \int_{-A}^{A} \abs{\phi_n(t) - e^{-t^2/2}} \d t
		+ \int_{A \leq \abs{t} \leq \pi \sigma} (\abs{\phi_n(t)} + |e^{-t^2/2}|) \d t.
\]

Since we have a CLT, the first integral on the right goes to 0. Hence, we have reduced the problem of proving an LLT to one of sufficiently bounding the characteristic function $\abs{\phi_n(t)}$. To get a quantitative LLT, we must also bound $\abs{\phi_n(t) - e^{-t^2/2}}$ for small values of $t$.
The overall strategy used to bound the $\abs{\phi_n(t)}$ involves finding an event that occurs with high probability and allows $R_n$ to be written as the sum of $n$ i.i.d. random variables $X_i$ after conditioning on the event. The full proof can be found in \cite{GilmerKopparty14}.

Using the $p$-biased Fourier basis, Berkowitz \cite{Berkowitz17} shows $\int_{-\pi\sigma_n}^{\pi\sigma_n} \abs{\phi_n(t) - e^{-t^2/2}} \d t = O(n^{-1/2+\epsilon}$). In section 4 below, we set up and apply this tool for 3-term arithmetic progressions in order to bound $\abs{\phi_n(t) - e^{-t^2/2}}$ for small values of $t$, although we conjecture that there is no bound for larger $t$.

\section{Descents in a permutation}

Let $\pi \in S_n$ be a permutation. Define $D(\pi) = \abs{\{1 \leq j < n \mid \pi(j + 1) < \pi(j)\}}$ to be the number of descents in $\pi$. Viewing $D_n$ as a random variable on uniformly distributed permutations, a central limit theorem is known \cite{Fulman97} and in this section, we prove Theorem \ref{thm:llt-dn}, a local limit theorem. In the proof, we will apply the following bound.

\begin{lem}{(Gilmer, Kopparty; 2014)} \label{gkbound}
Let B be a Bernoulli random variable that is 1 with probability $p$. Then 
\[ \abs{e^{i\theta B}} \leq 1-8p(1-p) \cdot \left\lVert\frac{\theta}{2\pi}\right\rVert^2
\]
\noindent where $\norm{x}$ is the closest integer to $x$. As a result, for $\abs{\theta} < \pi$,
\[ \abs{e^{i\theta B}} \leq 1-8p(1-p) \cdot \p{\frac{\theta}{2\pi}}^2. \]
\end{lem}

As covered in \ref{descents-intro}, we write $D_n = \sum_{j=1}^{n-1} X_j$, where $X_j$ is the indicator random variable for $a_j > a_{j+1}$. Each $X_j$ depends only on $a_j$ and $a_{j+1}$, so $X_j$ is independent of all other $X_k$ except for $X_{j-1}$ and $X_{j+1}$. This lemma computes the dependence:
\begin{lem}
	Given the values of $X_{j-1}$ and $X_{j+1}$, the distribution of $X_j$ does not depend on $j$ or $n$.
    \proof
    The descents $X_{j-1}$ to $X_{j+1}$ depend only on $a_{j-1}$ to $a_{j+2}$. Every combination of values of $a_{j-1}$ to $a_{j+2}$ is equally likely, so it suffices to consider the case $n = 4$, $j = 2$.
    
    Although it is not necessary for the local limit theorem, the exact distribution of $X_j$ can be computed. In the $n = 4$ case, each pair of values of $X_1$ and $X_3$ appears in exactly $6$ permutations. The following probabilities are established by looking at all $6$ cases for each pair.
    \begin{align*}
		\P(X_j = 1 \mid X_{j-1} = 1, X_{j+1} = 1) &= 1/6, \\
		\P(X_j = 1 \mid X_{j-1} = 1, X_{j+1} = 0) &= 1/2, \\
		\P(X_j = 1 \mid X_{j-1} = 0, X_{j+1} = 1) &= 1/2, \\
		\P(X_j = 1 \mid X_{j-1} = 0, X_{j+1} = 0) &= 5/6.
	\end{align*}
    \qed
\end{lem}
\begin{thm}
The sequence of random variables $D_n$ satisfies a local limit theorem. Quantitatively,
\[ \P(D_n = x) = \frac{1}{\sqrt{2\pi} \sigma_n} e^{-((x-\mu_n)/\sigma_n)^2/2} + O(n^{-1+\eps}). \]
\proof
	To bound the characteristic function of $D_n$, we condition on the values of $X_j$ for odd $j$. For simplicity, assume $n$ is even. (When $n$ is odd, one must also observe that $\P(X_j = 1 \mid X_{j-1} = 1) = 1/3$ and $\P(X_j = 1 \mid X_{j-1} = 0) = 2/3$ are constant.) The even $X_j$ are independent of each other, so conditioned on the odd $X_j$, $D_n$ is a sum of independent random variables. After conditioning, we compute
    
   	\begin{align*}
		\abs{\E[e^{itD_n/\sigma}]}
		&= \abs{\E[e^{it(C + \sum\limits_{\text{odd } j} X_j)/\sigma}]} \\
		&= \prod_{\text{odd } j} \abs{\E[e^{itX_j/\sigma}]} \\
        &\leq \prod_{\text{odd } j} \p{1 - 8p_j(1-p_j)\p{\frac{t}{2\pi\sigma}}^2}
	\end{align*}
    by Lemma \ref{gkbound}, where $p_j$ is the probability associated to $X_j$ by the given values of the odd $X_j$. But there are only finitely many possible values of $p_j$ (4 of them), so there is some value $p$ that maximizes the quantity above, and we can bound
	\begin{align*}
		\abs{\E[e^{itD_n/\sigma}]}
        &\leq \prod_{\text{odd } j} \p{1 - 8p(1-p)\p{\frac{t}{2\pi\sigma}}^2} \\
		&\leq \prod_{\text{odd } j} (1 - \Theta(t^2/n)) \\
		&= (1 - \Theta(t^2/n))^{n/2} \\
		&\leq e^{-\Theta(t^2)/2} \\
		&= e^{-\Theta(t^2)}.
	\end{align*}

	Now we obtain the final bound for the local limit theorem. As shown in $\cite{Fulman97}$, we have a central limit theorem 
\[ \sup_{ -\infty < x < \infty} \abs{\P\p{\f{D_n - \mu}{\sigma} < t} - \P(Z < t)} \leq \sqrt{\f{12}{n}}. \]

For $y<0$, $\P(D_n-\mu \leq y\sigma) = \frac{1}{2} \P(\abs{D_n-\mu} \geq \abs{y}\sigma) \leq \frac{1}{2y^2}$ by Chebyshev's Inequality. Similarly, $\P(D_n-\mu \geq y\sigma) \leq \frac{1}{2y^2}$ for $y>0$. 

Thus, \[\abs{\P\p{\f{D_n - \mu}{\sigma} < t} - \P(Z < t)} = \begin{cases} 
      \f{12}{\sqrt{n}} \\[10pt]
      \P(\f{D_n - \mu}{\sigma} \leq y) + \P(Z \leq y) \leq \frac{1}{2y^2} + O(e^{-\Theta(y^2)}/\abs{y}) & \text{for } y < 0 \\[10pt]
      \P(\f{D_n - \mu}{\sigma} \geq y) + \P(Z\geq y) \leq \frac{1}{2y^2} + O(e^{-\Theta(y^2)}/\abs{y}) & \text{for } y > 0 \\[10pt]
\end{cases}. \]

For a characteristic function $\phi(t)$ for a variable $X$, using integration by parts we can write
\[
\phi(t)
= \int_\R e^{itx} \P(X = x)\ \d x
= -it \int_\R e^{itx} \P(X \leq x)\ \d x.
\]

Hence, we have the following (where $\abs{\ldots}$ is a repetition of the same integrand):
\begin{align*}
\abs{\phi_n(t) - e^{-t^2/2}} &\leq \abs{t} \int_\R \abs{\P\p{\f{D_n - \mu}{\sigma} < y} - \P(Z < y)} \d y  \\
&\leq |t| \int_{|y|>k\sigma} \abs{\P\p{\f{D_n - \mu}{\sigma} < y} - \P(Z < y)} \d y + |t| \int_{|y| \leq k\sigma} \abs{\ldots} \d y \\
&\leq |t| \int_{|y|>k\sigma} \p{\frac{1}{2y^2} + \frac{e^{-\Theta(y^2)}}{y}} \d y + |t|2k\sigma \sqrt{\frac{12}{n}}.
\end{align*}

Take $k = O(n^{-\frac{1}{2} + \epsilon})$. Then since $\sigma = \Theta(n)$, we have $\abs{\phi_n(t) - e^{-t^2/2}} \leq |t|O(n^{-\frac{1}{2} + \epsilon})$.

	Thus, for any $\eps > 0$, we compute
	\begin{align*}
		\int_{-\pi \sigma}^{\pi \sigma} \abs{\phi_n(t) - e^{-t^2/2}} \d t
		&\leq \int_{-n^\eps}^{n^\eps} \abs{\phi_n(t) - e^{-t^2/2}} \d t + \int_{n^\eps < \abs{t} < \pi \sigma} (\abs{\phi_n(t)} + |e^{-t^2/2}|) \d t \\
		&\leq \int_{-n^\eps}^{n^\eps} \abs{\phi_n(t) - e^{-t^2/2}} \d t + \int_{n^\eps < \abs{t} < \pi \sigma} e^{-\Theta(t^2)} \d t \\
		&= O(n^{-\f{1}{2} + \eps}).
	\end{align*}
    \qed
    \end{thm}

\section{3-term arithmetic progressions}

Throughout this section, we define a random variable $A_n$ to be the number of 3-term arithmetic progressions in a randomly chosen subset $S \subset \Z / n\Z$, where each element in $\Z / n\Z$ has probability $\f{1}{2}$ (also denoted $p$) of appearing in the subset. The elements of the probability space can thus be described as $\bfx \in \{0,1\}^n$ where $\bfx_i$ is 1 with probability $p$ and 0 with probability $1-p$. Further, the random variable $A_n$ can be thought of as a function $A_n : \{0,1\}^n \to \bbn$.

Further, we require $n$ to be prime because we want a sort of uniformity among properties of arithmetic progressions. For example, the set $\{0, 5, 10, 15, 20\}$ in $\Z / 25\Z$ has $10$ arithmetic progressions, while no other set of size 5 has as many arithmetic progressions as this. To give a more concrete example of why this might end up messing with some calculations later on, consider the probability that a fixed arithmetic progression appears in a random subset of $\Z /n\Z$. If $n$ is prime, 

\subsection{Variance of $A_n$}
In this section, we describe a $p$-biased Fourier basis for functions on the probability space, exactly analogous to the method of Berkowitz \cite{Berkowitz17}, which we use to find the variance of $A_n$ and will use in the next section to try to prove an LLT for $A_n.$
In order to do this, we first define $\chi_i : \{0,1\}^n \to \bbr$ by

\[\chi_i := \chi_i(\bfx) := \frac{\bfx_i - p}{\sqrt{p(1-p)}} = \begin{cases} 
      -\sqrt{\frac{p}{1-p}} & \text{if } \bfx_p = 0 \\[10pt]
      \sqrt{\frac{1-p}{p}} & \text{if } \bfx_p = 1
\end{cases} \]
so that $\chi_i$ is a normalized version of $\bfx_i$. Further, we can extend this to define, for an arbitrary set $S \subseteq \Z / n\Z$,

\[\chi_S := \prod_{i \in S} {\chi_i}. \]

Note that if we take the inner product of two functions $f, g : \{0,1\}^n \to \bbr$ to be $\E[fg]$, then $\{\chi_S \mid S \subseteq \Z / n\Z\}$ forms an orthonormal basis for functions on our probability space. Then if we define the Fourier transform $\hat{f} : \{0,1\}^n \to \bbr$ of an arbitrary function $f : \{0,1\}^n \to \bbr$ by

\[\hat{f}(S) := \E[f(\bfx)\chi_S(\bfx)],\]
from the orthonormality of our basis we get

\[ f(\bfx) = \sum_{S \subseteq \Z / n\Z} {\hat{f}(S)\chi_S(\bfx)}. \]

We will use this expansion to calculate the variance of $A_n$ and bound the pointwise distance of the characteristic function from that of the discrete Gaussian for small $t$.

It will now be useful to normalize the random variable $A_n$. We take the mean of $A_n$ to be $\mu_n := \E[A_n]$. We write the variance of $A_n$ as $\sigma_n^2 := \E[A_n^2] - \E[A_n]^2$. So we define $Z : \{0,1\}^n \to \bbr$ by

\[Z = Z_n := \frac{A_n - \mu_n}{\sigma_n}\]
and we will often refer to the characteristic function of $Z$ defined by $\phi_Z(t) := \E[e^{itZ}]$.

Before moving on to calculate the Fourier coefficients $\hat{A_n}(S)$, we will first note that there are $\binom{n}{2}$ possible (non-trivial) 3-term arithmetic progressions in $\Z / n\Z$. There are first $n$ choices for the start of the arithmetic progression, then $n-1$ choices for a non-trivial separation distance $d$, and finally both $d$ and $-d$ will have counted the same 3-term arithmetic progression from different starting points, so we divide by 2 to yield $\frac{n(n-1)}{2} = \binom{n}{2}$. Additionally, each 3-term arithmetic progression occurs with probability $p^3$ (each of the three terms in the progression occur independently with probability $p$). This allows us to calculate

\[\mu_n = \E\s{\sum_{\Lambda} {1_\Lambda}} = \sum_{\Lambda} {\E[1_\Lambda]} = \sum_{\Lambda} {p^3} = p^3\binom{n}{2}\]
where $1_\Lambda$ is the indicator function for a particular 3-term arithmetic progression $\Lambda$ in $\Z / n\Z$.

Furthermore, the fact that our basis for functions on the probability space is orthonormal allows us to calculate variance according to this formula, derived from Parseval's Theorem.

\[ \sigma^2 = \sum_{S \neq \varnothing} {\hat{A_n}(S)^2}. \]

We now work with the 3-term arithmetic progression indicator functions $1_\Lambda$ a bit more. We use $i \in \Lambda$ to denote that $i$ is a term in the 3-term arithmetic progression $\Lambda$. Therefore, the indicator function can be expressed as

\begin{align*}
1_\Lambda(\bfx) &= \prod_{i \in \Lambda} {\bfx_i} = \prod_{i \in \Lambda} {\p{\sqrt{p(1-p)}\chi_i + p}} \\ 
&= p^3+p^2\sqrt{p(1-p)}\sum_{i \in \Lambda} {\chi_i} + p^2(1-p)\sum_{i_1 \neq i_2 \in \Lambda} {\chi_{\{i_1, i_2\}}} + p^{3/2}(1-p)^{3/2}.
\end{align*}

Note that any two elements of $\Z / n\Z$ appear in exactly 3 3-term arithmetic progressions and any one element appears in exactly $\frac{3}{2}(n-1)$ 3-term arithmetic progressions. Hence, by summing over all 3-term arithmetic progressions, we have
\begin{align*}
A_n = p^3 {{n}\choose{2}} &+ \frac{3}{2}(n-1)p^2\sqrt{p(1-p)}\sum_{i \in \Z / n\Z} \chi_i \\
&+3p^2(1-p)\sum_{\{i_1, i_2 \mid i_1 \neq i_2\} \subset \Z / n\Z} \chi_{\{i_1, i_2 \}} + \sum_{\Lambda} p^{3/2}(1-p)^{3/2}\chi_{\Lambda}.
\end{align*}

Thus, we have the Fourier Transform of $A_n$:
\[\hat{A_n}(S) = \begin{cases} 
      p^3 {{n}\choose{2}} & \text{if } S = \varnothing \\[10pt]
      \frac{3}{2}(n-1)p^2\sqrt{p(1-p)} & \text{if } \abs{S} = 1 \\[10pt]
      3p^2(1-p) & \text{if } S = \{i_1,i_2\} \\[10pt]
      p^{3/2}(1-p)^{3/2} & \text{if } S = \Lambda \\[10pt]
      0 & \text{else}
\end{cases} \]
and we can use Parseval's Theorem to give us the variance of $A_n$:

\begin{align*}
\sigma_n^2 &= \sum_{\substack{S \subset \Z / n\Z \\ S \neq \varnothing}} {\hat{A_n}(S)^2} \\
&= \sum_{i \in \Z / n\Z} {\p{\frac{3}{2}(n-1)p^2\sqrt{p(1-p)}}^2} + \sum_{\{i_1, i_2\} \in \Z / n\Z} {\p{3p^2(1-p)}^2} + \sum_{\Lambda} {\p{p^{3/2}(1-p)^{3/2}}^2} \\
&= \frac{9}{4}n(n-1)^2p^5(1-p) + 9\binom{n}{2}p^4(1-p)^2 + \binom{n}{2}p^3(1-p)^3 \\
&= \Theta(n^3).
\end{align*}

This means that $\sigma_n = \Theta(n^{3/2})$.

\subsection{A CLT for 3-term APs}

In this section, we establish a central limit theorem for the number of 3-term arithmetic progressions in $\Z / n\Z$, $A_n$.

\begin{thm}
\label{thm:clt-an}
$$\abs{P\p{\frac{A_n - \mu_n}{\sigma_n} \leq x} - \Phi(x)} = O\p{n^{-1/4}}.$$
\end{thm}

We prove this theorem using the method of dependency graphs detailed in Chatterjee's notes \cite{Chatterjee07} (Lecture 6). This method bounds the Wasserstein distance of a distribution from the normal based on the dependency graph. We define the Wasserstein distance as such, where $\Omega$ is the probability space:

$$ \Wass(\mu, \nu) = \sup\left\{{\abs{\int{f \d \mu} - \int{f \d \nu}}\ \big|\ f : \Omega \to \R \text{ is 1-Lipschitz}}\right\}.$$

The Wasserstein distance is a metric that is stronger than the $L_1$ distance of the cumulative distribution functions of two distributions \cite{Chatterjee07}. In other words, a vanishing bound on the Wasserstein distance between a distribution and the standard normal would immediately establish a central limit theorem for that distribution.

More formally, we define the Kolmogorov distance as:

$$ \Kolm(\mu, \nu) = \sup_{x \in \R}{\abs{\mu((-\infty, x]) - \nu((-\infty, x])}}. $$

So the Kolmogorov distance is the largest difference in the cumulative distribution functions. We write the Wasserstein distance between two random variables $X \sim \mu$ and $Y \sim \nu$ as $\Wass(X, Y) = \Wass(\mu, \nu)$, and similarly, $\Kolm(X, Y) = \Kolm(\mu, \nu)$. A useful bound for the Kolmogorov distance between a random variable and the standard normal is stated and proved in \cite{Chatterjee07}:

\begin{lem}[Chatterjee, 2007]
\label{lem:wass-kolm}
For a pair of r.v.'s $W, Z$ where $Z \sim N(0, 1)$,
$$ \Kolm(W, Z) \leq \sqrt{\frac{2}{\pi}\ \Wass(W, Z)}. $$
\end{lem}

Thus, the distance between the cdf of a random variable $W$ and the standard normal, $Z$ is bounded by the square root of the Wasserstein distance between $W$ and $Z$ (with a constant factor). So the problem of deriving a central limit theorem for $W$ is reduced to bounding $\Wass(W, Z)$.

A formal statement of the method of dependency graphs is as follows. Suppose there are a collection of random variables, $\{X_i \mid i \in V\}$, indexed by the vertices of a graph $G = (V,E)$ such that $(i, j) \in E$ iff $X_i, X_j$ are dependent. Such a graph is referred to as a dependency graph. Let $D = 1 + \Delta(G)$, where $\Delta(G)$ is the maximum degree of $G$. We have the following:

\begin{lem}[Chatterjee, 2007]
\label{lem:wass-dep-graph}
Suppose that $\E[X_i] = 0$, $\sigma^2 = \Var(\sum {X_i})$, $W = \frac{\sum X_i}{\sigma}$ and $Z \sim N(0,1)$. Then,

$$ \Wass(W, Z) \leq \frac{4}{\sqrt{\pi}\sigma^2} \sqrt{D^3 \sum {\E{|X_i|}^4}} 
+ \frac{D^2}{\sigma^3} \sum {\E{|X_i|}^3}. $$

\end{lem}

For the problem of 3-term arithmetic progressions, we will write $\Lambda$ to denote a particular 3-term AP. We define $V = \{\Lambda \subseteq \Z /n\Z \mid \Lambda \text{ is a 3-AP}\}$ and $X_\Lambda = 1_{\Lambda \in S} - \E[1_{\Lambda \in S}] = 1_{\Lambda \in S} - p^3$. Note that there are $\binom{n}{2}$ 3-term APs in $\Z / n\Z$, not counting degenerate 3-term arithmetic progressions, so $|V| = \binom{n}{2}$. Further, note that $X_{\Lambda_1}$ and $X_{\Lambda_2}$ are independent iff $\Lambda_1 \cap \Lambda_2 = \varnothing$. We have the following properties:

\begin{lem}
The degree of the dependency graph is $D = O(n)$.
\end{lem}

\begin{proof}

\begin{align*}
\deg(\Lambda_1) &= \abs{\{ \Lambda_2 \mid \Lambda_1 \cap \Lambda_2 \neq \varnothing \}} \\
&= \abs{\{ \Lambda_2 \mid (\Lambda_2)_1 \in \Lambda_1 \lor (\Lambda_2)_2 \in \Lambda_1 \lor (\Lambda_2)_3 \in \Lambda_1\}} \\
&\leq 3\abs{\{ \Lambda_2 \mid (\Lambda_2)_1 \in \Lambda_1 \}} \\
&= \frac{9}{2}(n-1).
\end{align*}

So $\Delta(G) \leq \frac{9}{2}(n-1)$ and $D \leq \frac{9}{2}(n-1)+1$. Thus, $D = O(n)$.
\end{proof}

\begin{prop}
The $m$th moment of $X_\Lambda$ is bounded by $\E|X_\Lambda|^m \leq 1$ for all $m \geq 1$.
\end{prop}

\begin{proof}

$X_\Lambda$ is a random variable which takes values at $1-p^3$ and $p^3$, so its absolute value is bounded by 1. Thus, its $m$th absolute moment is also bounded by 1.

\end{proof}

Furthermore, we note that $\Var\p{\sum X_i} = \Var\p{\sum 1_\Lambda - p^3} = \sigma_n^2$ and $W = \frac{\sum X_i}{\sigma} = \frac{A_n - \mu_n}{\sigma_n}$. We now proceed with the proof of theorem \ref{thm:clt-an}.

\begin{proof}(Theorem \ref{thm:clt-an})

Applying lemma \ref{lem:wass-dep-graph} to $W = \frac{A_n - \mu_n}{\sigma_n}$ and using the bound on the Kolmogorov distance from \ref{lem:wass-kolm} we get:

\begin{align*}
\abs{P\p{\frac{A_n - \mu_n}{\sigma_n} \leq x} - \Phi(x)} &\leq \sqrt{\frac{2}{\pi} \Wass(W, Z)} \\
&\leq \sqrt{\frac{2}{\pi}}\sqrt{\frac{4}{\sqrt{\pi}\sigma^2} \sqrt{D^3 \sum {\E{|X_\Lambda|}^4}} 
+ \frac{D^2}{\sigma^3} \sum {\E{|X_\Lambda|}^3}}
\end{align*}

Now, using the established properties that $|V| = \binom{n}{2} \leq n^2$, $D = O(n)$, $\E{|X_\Lambda|^m} \leq 1$ for all $m \geq 1$, and $\sigma^2 = \Theta(n^3)$, we get

\begin{align*}
\abs{P\p{\frac{A_n - \mu_n}{\sigma_n} \leq x} - \Phi(x)} &= O\p{\sqrt{\frac{1}{n^3} \sqrt{n^3 n^2} + \frac{n^2}{n^{9/2}} n^2}} \\
&= O(n^{-1/4}).
\end{align*}

\end{proof}


\subsection{Attempt to show LLT for $A_n$}
We recall from section 2.1.1 that since we have a CLT, we can reduce the problem of proving an LLT to one of bounding $\abs{\phi_Z(t) - e^{-t^2/2}}$. 
Using the $p$-biased Fourier basis developed above, we proceed to show the following bound for small values of $t$ using the same method as Berkowitz \cite{Berkowitz17}: 
\begin{prop}
For $|t| \ll \sqrt[]{n}$, \[ \abs{\phi_Z(t) - e^{-t^2/2}} \leq O\left(\frac{t^3e^{-t^2/3}}{\sqrt[]{n}} + \frac{t}{\sqrt{n}}\right). \]
\end{prop}

\begin{proof}
To begin, we decompose $Z = X+Y$, where
\[Q = \frac{1}{\sqrt{n}}, X = \sum_{k \in \Z / n\Z} Q\chi_k, Y = \sum_{k \in \Z / n\Z} (\hat{Z}(k)-Q)\chi_k + \sum_{|S| \geq 2} \hat{Z}(k)\chi_k.\]

We view $X$ as the main term and $Y$ as the error term. Since we normalized $X$ with $Q$, $X$ has mean $0$ and variance $1$. We first bound the distance between $X$ and the normal distribution. Define
\[L_n := n\E[|Q\chi_k|^3] = O\left(\frac{1}{\sqrt{n}}\right) < \infty. \]
By Berry-Esseen, for $t \leq \f{1}{4L_n}$ we have
\[|\E[e^{itX}] - e^{-t^2/2}| \leq 16L_{n}t^3e^{-t^2/3}.\]

Bounding the expectation of $Y$, we obtain
\[ \E[|Y|] \leq \sqrt{\E[|Y|^2]} = \sqrt{\Var(Y)} = \sqrt{\sum_{k \in \Z / n\Z} (\hat{Z}(k)-Q)^2 + \sum_{|S| \geq 2} \hat{Z}^2(S)}.\]
Using the value of the variance $\sigma_n^2$ of $A_n$, \[\sum_{|S| \geq 2} {\hat{Z}(S)} = \frac{9{{n}\choose{2}}p^4(1-p)^2+{n\choose2}p^3(1-p)^3}{\sigma_n^2} = \Theta\left(\frac{1}{n}\right).\]
In addition, using the value of $\hat{A_n}$,
\begin{align*}
n\sigma_n\hat{Z}^2(k) - \sigma_n^2 = n\hat{A_n}^2(k) - \sigma_n^2 = O(n^2)  \\
n\hat{Z}^2(k) - 1 = O\left(\frac{1}{n} \right) \\
\hat{Z}^2(k) - \frac{1}{n} = O\left(\frac{1}{n^2} \right).
\end{align*}
Since $\hat{Z}(k) + Q = O\left(\frac{1}{\sqrt{n}}\right)$ and $Q^2 = \frac{1}{n}$,
\[ \abs{\hat{Z}(k) - Q} \leq \abs{\frac{\hat{Z}(k) - \frac{1}{n}}{\hat{Z}(k)+Q}} = O\p{\frac{1}{n^{3/2}}}. \]
Hence, we have $\E[Y] = O\left(\frac{1}{\sqrt{n}}\right)$. 
Thus, we conclude that if $t < \frac{1}{4L_n} = O(\sqrt{n})$,
\begin{align*} 
\abs{\phi_Z(t) - e^{-t^2/2}}
&= \abs{\E[e^{itZ}-e^{-t^2/2}]} = \abs{\E[e^{it(X+Y)}-e^{-t^2/2}]} \\
&\leq \abs{\E[e^{it(X+Y)}-e^{itX}]} + \abs{\E[e^{itX}-e^{-t^2/2}]} \\
&\leq \E[tY] + 16L_{n}t^3e^{-t^2/3} \qquad \text{ (using the mean value theorem on } e^{itX})\\
&= O\left(\frac{t^3e^{-t^2/3}}{\sqrt[]{n}} + \frac{t}{\sqrt{n}}\right).
\end{align*}

\end{proof}

\begin{remark}
The bound we find for the absolute distance between the characteristic function of $Z$ and that of the standard normal is a decent bound for small $t$ if we want to show an LLT for $A_n$. However, we are unable to prove any meaningful bound for larger $t$, and below we explore the reasons for why this is impossible.
\end{remark}

\subsection{Experimental results and conjectures}

As we could not decrease the bound achieved in the previous section and prove a bound for larger values of $t$, we decided to simulate the number of 3-term arithmetic progressions in a random subset of $\Z / n \Z$ in order to verify whether or not an LLT actually holds and found that one does not.

\begin{figure}[t]
	\centering
	\includegraphics[width=\textwidth]{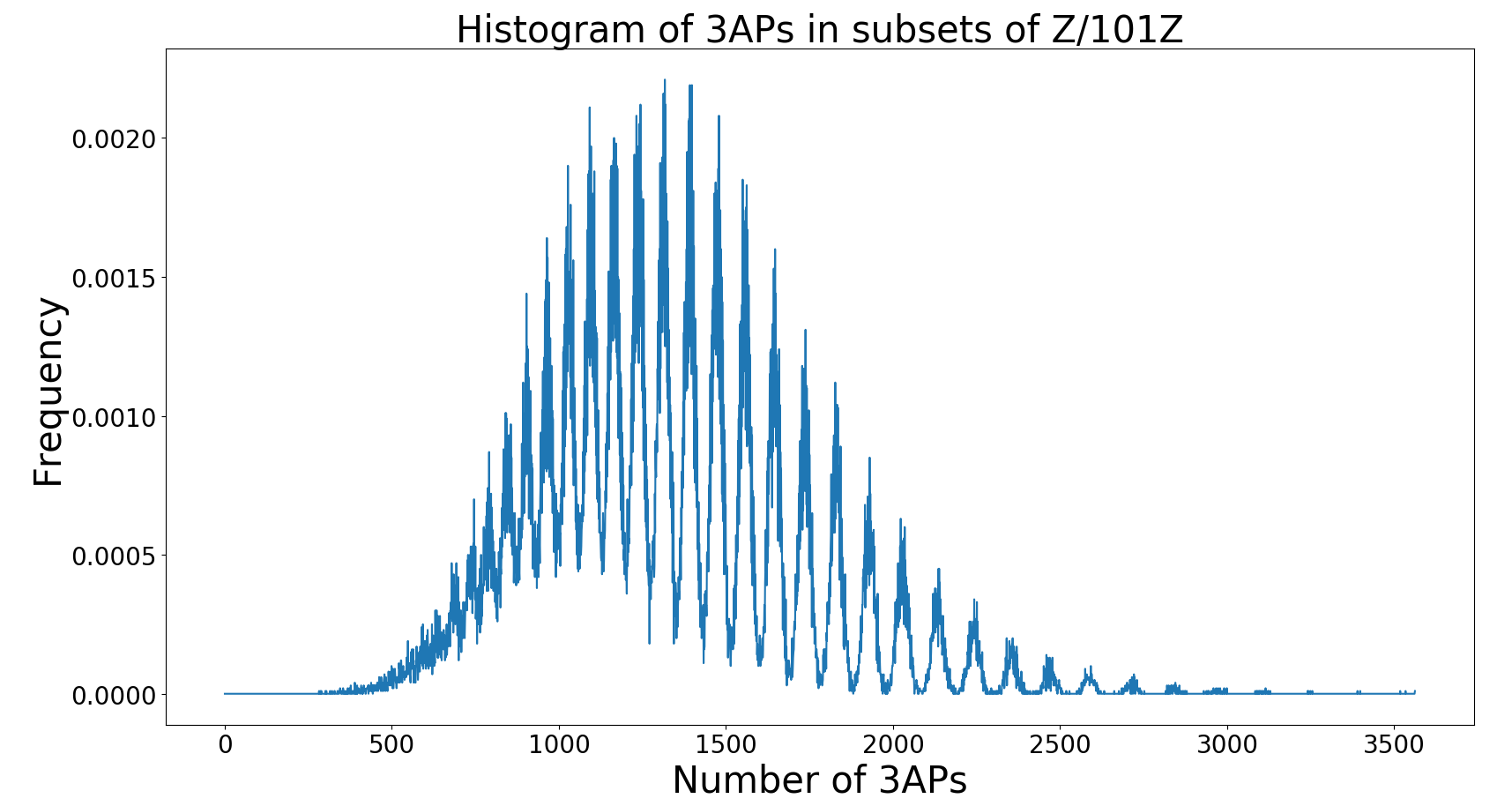}
    \caption{Histogram of the number of 3-term arithmetic progressions in 100000 random subsets of $\Z / 101\Z$.}
    \label{fig:3ap}
\end{figure}

As figure \ref{fig:3ap} demonstrates, the number of 3-term arithmetic progressions in a random subset of $\Z / n \Z$ does not follow the Gaussian distribution pointwise. However, after some inspection, one might notice that oscillations in the histogram that are close together are spaced out almost evenly. Furthermore, the oscillations themselves appear fairly Gaussian, as if the entire distribution is the sum of several spaced out Gaussians. Given that the spacings appear to be $O(n)$ and the entire distribution ranges from 0 to $\binom{n}{2} = O(n^2)$, we would expect there to be $O(n)$ of these smaller Gaussians. It is thus reasonable to conclude that the number of 3-term arithmetic progressions in a random subset of $\Z / n\Z$ depends on some other random variable which can take $O(n)$ values.

\begin{figure}[t]
	\centering
    \includegraphics[width=\textwidth]{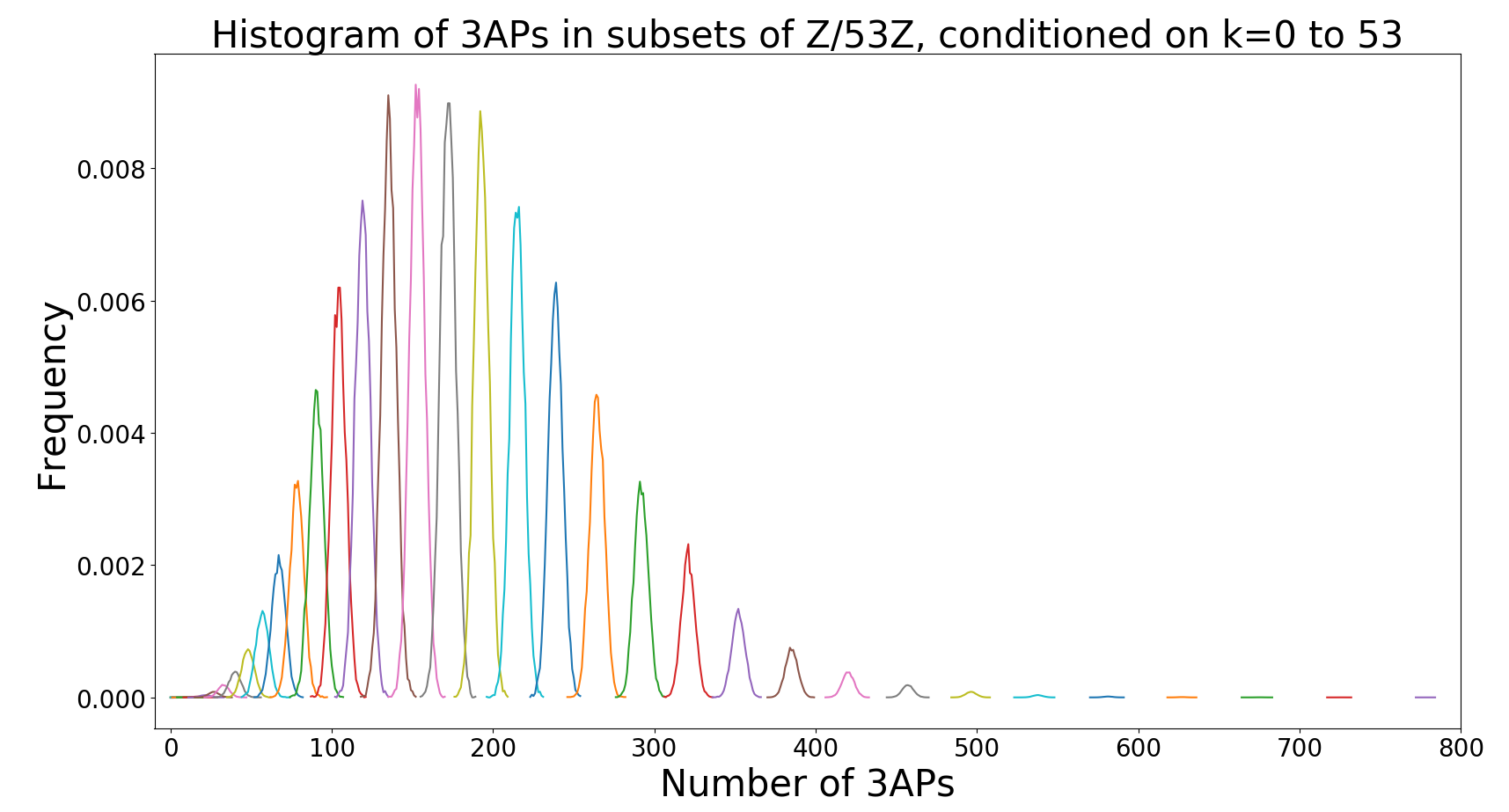}
    \caption{Histogram of the number of 3-term arithmetic progressions in random subsets of $\Z/53\Z$ of size $k$ for $k=0$ to $53$, separated by color and . Each $k$ used 10000 samples. (The graph should extend to 1378 on the x axis, but there is very little content in the large $k$ that are very far out so we zoomed in on the bulk of the distribution.)}
    \label{fig:3ap-k}
\end{figure}

This other random variable is the size of a random set of $\Z / n\Z$. From figure \ref{fig:3ap-k} we can see that the number of 3-term arithmetic progressions in a random subset of $\Z / n \Z$ with fixed size does, in fact, follow the Gaussian distribution pointwise. From now on, we call this random variable $A_{n, k}$, where $k$ is the size of the subset of $\Z / n\Z$.

We conjecture that the distance between these smaller distributions, $\E[A_{n,k+1}]-\E[A_{n,k}]$, is sufficiently smaller than the standard deviation of the distributions themselves, $\sqrt{\Var(A_{n, k})}$. In other words, there is a non-trivial interval between $\E[A_{n, k}]$ and $\E[A_{n, k+1}]$ where for all x in the interval, $\P(A_{n, k}=x)$ and $\P(A_{n, k+1} = x)$ are both very small. In particular, we suspect that those two probabilities are significantly smaller than the Gaussian approximation of $\P(A_n = x)$ would suggest.

This train of thought will be made more precise in sections 6 and 7.

\subsubsection{Continuous sets}

One thing that comes to mind after observing that $A_n$ does not have an LLT due to the dependence on the size of the random set, is whether $A_n$ would have an LLT if the existence of each element of the set was made continuous. That is to say, the variable $x_i$ which denotes whether an element $i$ of $\Z / n\Z$ is in our random set $S$, is now a uniform random variable on $[0, 1]$ instead of a Bernoulli random variable on $\{0, 1\}$ with $p=\frac{1}{2}$. The number of 3-term arithmetic progressions $\bar{A}_n$ is still defined in a similar manner,

$$ \bar{A}_n = \frac{1}{2}\sum_{i = 1}^{n} {\sum_{j = 1}^{n-1} {x_i x_{i+j} x_{i+2j}}}. $$

\begin{figure}[t]
	\centering
    \includegraphics[width=\textwidth]{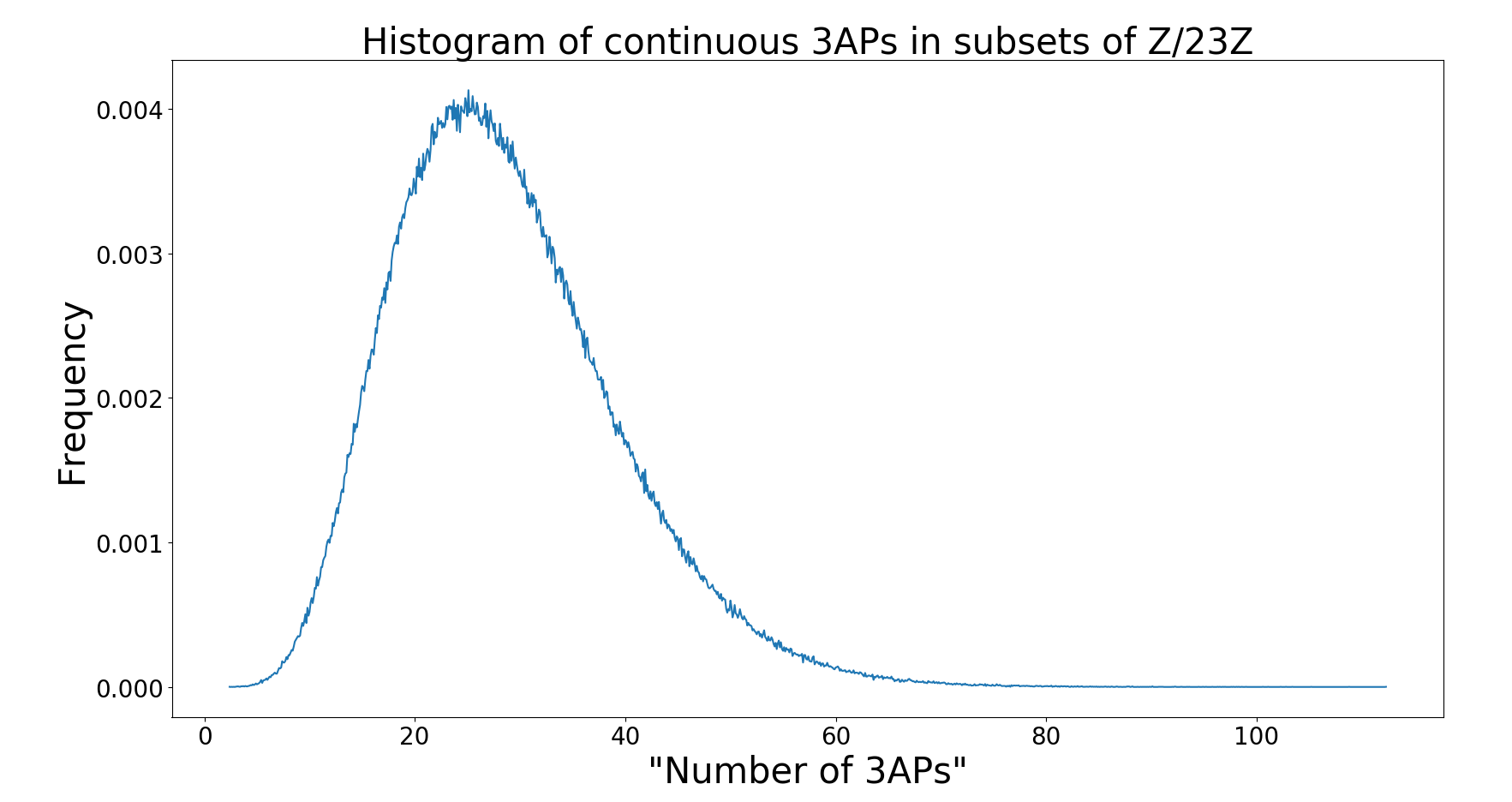}
    \caption{Number of 3-term "arithmetic progressions" in 1000000 random "continuous sets" of $\Z / 23\Z$} 
    \label{fig:3ap-cont}
\end{figure}

We simulated this to see if it possibly followed the Gaussian distribution pointwise. From figure \ref{fig:3ap-cont}, we conjecture that it is very likely that there is an LLT for $\bar{A}_n$. However, we successfully can show that there is a CLT for $\bar{A}_n$. Namely, define $\bar{\mu}_n = \E[\bar{A}_n]$ and $\bar{\sigma}_n^2 = \Var(A_n)$. Then the following theorem holds:

\begin{thm}
\label{thm:clt-an-cont}
$$\abs{P\p{\frac{\bar{A}_n - \bar{\mu}_n}{\bar{\sigma}_n} \leq x} - \Phi(x)} = O\p{n^{-1/4}}.$$
\end{thm}

Notably, this is exactly the same bound achieved in Theorem \ref{thm:clt-an}. This CLT is a consequence of the fact that the dependency graph of the indicator functions $1_\Lambda$ does not change if we make the underlying $x_i$ uniform on [0, 1] instead of Bernoulli on \{0, 1\}. However, the following lemma is what gives us the exact same bound as for the standard case, $A_n$:

\begin{lem}
\label{lem:var-an-cont}
$\bar{\sigma}_n^2 = \Theta(n^3).$
\end{lem}

\begin{proof}
First, note that 
$$\bar{\mu}_n = \E\left[\sum_{\Lambda} {1_\Lambda}\right] = \sum_{\Lambda} {\E[1_\Lambda]} = \frac{1}{8} \binom{n}{2}.$$
Now, we can also calculate
\begin{align*}
\E\left[ \bar{A}_n^2 \right] &= \E\left[\p{\sum_{\Lambda} {1_\Lambda}}^2\right] = \E\left[\sum_{\Lambda_1} {\sum_{\Lambda_2} {1_{\Lambda_1} 1_{\Lambda_2}}}\right] \\
&= \E\left[\sum_{\Lambda_1} {\sum_{\Lambda_2} {\prod_{j \in {\Lambda_1 \cup \Lambda_2}} {x_j}}}\right] \\
&= \E\left[\sum_{\Lambda_1} {\sum_{i=0}^{3} {\sum_{\substack{\Lambda_2 \textrm{ s.t. } \\ \abs{\Lambda_1 \cap \Lambda_2} = i}} {\prod_{j \in {\Lambda_1 \cup \Lambda_2}} {x_j}}}}\right] \\
&= \sum_{\Lambda_1} {\sum_{i=0}^{3} {\sum_{\substack{\Lambda_2 \textrm{ s.t. } \\ \abs{\Lambda_1 \cap \Lambda_2} = i}} {\p{\frac{1}{2}}^{6-i}}}} \\
&= \sum_{\Lambda_1} {\sum_{i=0}^{3} {\p{\frac{1}{2}}^{6-i} \p{\# \{\Lambda_2 \mid \abs{\Lambda_1 \cap \Lambda_2} = i \}}}}.
\end{align*}

The particular quantity $\# \{\Lambda_2 \mid \abs{\Lambda_1 \cap \Lambda_2} = i \}$ is examined further and calculated in section \ref{var-ank}. For now we take the particular values we need as given and note that the quantity does not depend on our choice of $\Lambda_1$.

$$
\E\left[ \bar{A}_n^2 \right] = \binom{n}{2} \p{\frac{1}{8} + \frac{1}{16}(6) + \frac{1}{32}\p{\frac{9}{2}n-\frac{39}{2}} + \frac{1}{64}\p{\frac{1}{2}n^2 - 5n + \frac{25}{2}}}.
$$
Now, given that $\bar{\sigma}_n^2 = \E\left[ \bar{A}_n^2 \right] - \bar{\mu}_n^2$, we have
$$ \bar{\sigma}_n^2 = \frac{1}{64}\binom{n}{2}\p{n - \frac{25}{54}} = \Theta(n^3). $$

\end{proof}

\section{Attempts to get CLT for $A_{n,k}$}

In order to prove a CLT for $A_{n,k}$, we use the method of exchangeable pairs detailed in Chatterjee's notes \cite{Chatterjee07} (Lecture 7). Like the method of dependency graphs used in section 4, this method bounds the Wasserstein distance of a distribution from the normal distribution. 

\begin{defn}
$(W,W')$ is an exchangeable pair of random variables if $(W,W') \overset{\mathrm{d}}= (W',W)$.
\end{defn}

The following is a formal statement of the method of exchangeable pairs.

\begin{lem}{(Chatterjee, 2007)}
If $(W,W')$ is an exchangeable pair, $\E[W^2] = 1$ and there exists $\lambda \in (0,1)$ such that $\E[W-W'|W] = -\lambda W$, then 
\[\Wass(W,Z) \leq \sqrt{\p{\frac{2}{\pi}}\Var\p{\E\s{\frac{1}{2\lambda}(W'-W)^2 \mid W}}} + \frac{1}{3\lambda}\E[\abs{W'-W}^2],\]
where $Z \sim N(0,1)$.
\end{lem}

Let $A_{n,k}$ be the number of 3-term arithmetic progressions given the subset size = $k$, and $W$ be $A_{n,k}$ standardized. We want to show that $W$ converges in distribution to a normal distribution. We construct $W'$ as follows. We have a random subset $S \subseteq \Z / n\Z$ with size $k$. Randomly choose two elements $I,J \in \Z / n\Z$, swap their status regarding their inclusion in $S$, and call the new subset $S'$ (i.e. if $I \in S$ iff $J \in S'$ and $J \in S$ iff $I \in S'$). Now define $A$ to be the number of 3-term arithmetic progressions in $S$ and define $W$ to be standardized $A$, namely $\frac{A-\mu_{n,k}}{\sigma_{n,k}}$. Similarly, define $A'$ to be the number of 3-term arithmetic progressions in $S'$ and define $W'$ to be $\frac{A'-\mu_{n,k}}{\sigma_{n,k}}$.

\begin{remark}
Let $W$ and $W'$ be an exchangeable pair, each with mean 0 and variance 1, and let $\Lambda$ be a 3-term arithmetic progression in $\Z / n\Z$. Define $1_{\Lambda \subseteq S} = 
\begin{cases} 
      1 & \Lambda \subseteq S \\
      0 & \text{else}
\end{cases}$. \\
Then $A = \sum_{\Lambda \subseteq \Z / n\Z} 1_{\Lambda \subseteq S}$. Also note that
\[ \mu_{n,k} = \E[\sum_{\Lambda \subseteq \Z / n\Z} 1_{\Lambda \subseteq S}] = \sum_{\Lambda \subseteq \Z / n\Z} \P(\Lambda \in S) = \binom{n}{2}\frac{\binom{k}{3}}{\binom{n}{3}}. \]
\end{remark}

In order to apply the method, we first must show that $\E[W'-W \mid W] = -\lambda W$ for some $\lambda \in (0,1)$. 

\begin{lem}
Let $\lambda = \frac{3(n-k)}{\binom{n}{2}}$. Then $\E[W'-W \mid W]= -\lambda W.$
\end{lem}
\begin{proof}
We have 
\[\E[A' \mid A] = \sum_{\Lambda \subseteq \Z / n\Z} \E[1_{\Lambda \subseteq S'} \mid A] = \sum_{\Lambda \subseteq \Z / n\Z} \P(\Lambda \subseteq S' \mid A).\]

Now \[ \P(\Lambda \subseteq S' \mid A) = \P(\Lambda \subseteq S' \mid \Lambda \subseteq S, A)\P(\Lambda \subseteq S \mid A) + \P(\Lambda \subseteq S' \mid \Lambda \not \subseteq S, A)\P(\Lambda \not \subseteq S \mid A).
\]

Note that 
\[\P(\Lambda \subseteq S \mid A) = \frac{A}{\binom{n}{2}}\]
\[\P(\Lambda \not \subseteq S \mid A) = 1-\frac{A}{\binom{n}{2}}\]
as there are $A$ 3-term arithmetic progressions in $S$ and $\binom{n}{2}$ 3-term arithmetic progressions total in $\Z / n\Z$.

We have \[\P(\Lambda \subseteq S' \mid \Lambda \subseteq S, A) = 1- \frac{3(n-k)}{{{n}\choose{2}}}
\]
since the probability of $\Lambda \not \subseteq S'$ is the probability of one of its 3 elements being selected to be swapped along with one of the elements outside of $S$.

Taking this into account in our previous expression, we have
\[ \P(\Lambda \subseteq S' \mid A) = \p{1- \frac{3(n-k)}{{{n}\choose{2}}}}\p{\frac{A}{\binom{n}{2}}} + \P(\Lambda \subseteq S' \mid \Lambda \not \subseteq S, A)\p{1-\frac{A}{\binom{n}{2}}}. \]

Finally, we examine 
\begin{align*}
\P(\Lambda \subseteq S' \mid \Lambda \not \subseteq S, A) &= \sum_{i = 0}^{3} \P(\Lambda \subseteq S' \text{ and } \abs{\Lambda \cap S} = i \mid \Lambda \not \subseteq S \text{ and } A) \\
&= \sum_{i = 0}^{2} \P(\Lambda \subseteq S' \mid \Lambda \not \subseteq S \text{ and } \abs{\Lambda \cap S} = i  \text{ and } A)\P(\abs{\Lambda \cap S} = i \mid \Lambda \not \subseteq S \text{ and } A) \\
&= \P(\Lambda \subseteq S' \mid \abs{\Lambda \cap S} = 2, \text{ and } A)\P(\abs{\Lambda \cap S} = 2 \mid \Lambda \not \subseteq S\text{ and } A).
\end{align*}

We have $\P(\Lambda \subseteq S' \mid \abs{\Lambda \cap S} = 2, \text{ and } A) = \frac{k-2}{{n\choose2}}$ since the only way for $\Lambda$ to be contained in $S'$ if $\abs{\Lambda \cap S} = 2$ is if the one element of $\Lambda$ that is not in $S$ is chosen, along with one other element of $S \backslash \Lambda$.

Substituting this in, we get

\[\P(\Lambda \subseteq S' \mid \Lambda \not \subseteq S, A) = \p{\frac{k-2}{{n\choose2}}}\P(\abs{\Lambda \cap S} = 2 \mid \Lambda \not \subseteq S\text{ and } A).\]

\begin{align*}
\P(\abs{\Lambda \cap S} = 2 \mid \Lambda \not \subseteq S\text{ and } A) 
&= 3\P(\Lambda_1 \in S, \Lambda_2 \in S, \Lambda_3 \notin S \mid \Lambda \not \subseteq S \textrm{ and } A) \\
&= 3\P(\Lambda_3 \notin S \mid \Lambda_1 \in S, \Lambda_2 \in S, \Lambda \not \subseteq S, A) \\
&\hspace{10mm} \cdot \P(\Lambda_2 \in S \mid \Lambda \not \subseteq S, \Lambda_1 \in S, A)\P(\Lambda_1 \in S \mid \Lambda \not \subseteq S, A) \\
&= 3\P(\Lambda_2 \in S \mid \Lambda \not \subseteq S, \Lambda_1 \in S, A)\P(\Lambda_1 \in S \mid \Lambda \not \subseteq S, A).
\end{align*}

Here we have two non-trivial quantities to examine. First we see by Bayes's rule that 
\begin{align*}
\P(\Lambda_1 \in S \mid \Lambda \not \subseteq S, A) &= \dfrac{\P(\Lambda \not \subseteq S \mid \Lambda_1 \notin S, A)\P(\Lambda_1 \not \subseteq S,A)}{\P(\Lambda \not \subseteq S, A)} \\
&= \dfrac{\P(\Lambda \not \subseteq S \mid \Lambda_1 \notin S, A)\P(\Lambda_1 \notin S \mid A)}{\P(\Lambda \not \subseteq S \mid A)} \\
&= \dfrac{\P(\Lambda \not \subseteq S \mid \Lambda_1 \notin S, A)\p{\frac{k}{n}}}{1-\frac{A}{\binom{n}{2}}}.
\end{align*}

Further, we also have by Bayes's rule,
\begin{align*}
\P(\Lambda_2 \in S \mid \Lambda \not \subseteq S, \Lambda_1 \in S, A) &= \dfrac{\P(\Lambda \not \subseteq S \mid \Lambda_2 \in S, \Lambda_1 \in S, A)\P(\Lambda_2 \in S, \Lambda_1 \in S, A)}{\P(\Lambda \not \subseteq S, \Lambda_1 \in S, A)} \\
&= \dfrac{\P(\Lambda \not \subseteq S \mid \Lambda_2 \in S, \Lambda_1 \in S, A)\P(\Lambda_2 \in S \mid \Lambda_1 \in S, A)}{\P(\Lambda \not \subseteq S \mid \Lambda_1 \in S, A)} \\
&= \dfrac{\P(\Lambda_3 \notin S \mid \Lambda_2 \in S, \Lambda_1 \in S, A)\P(\Lambda_2 \in S \mid \Lambda_1 \in S, A)}{\P(\Lambda \not \subseteq S \mid \Lambda_1 \in S, A)} \\
&= \dfrac{\p{1-\frac{k-2}{n-2}}\p{\frac{k-1}{n-1}}}{\P(\Lambda \not \subseteq S \mid \Lambda_1 \in S, A)}.
\end{align*}

So,
\begin{align*}
\P(\abs{\Lambda \cap S} = 2 \mid \Lambda \not \subseteq S\text{ and } A) &= 3\p{\dfrac{\P(\Lambda \not \subseteq S \mid \Lambda_1 \notin S, A)\p{\frac{k}{n}}}{1-\frac{A}{\binom{n}{2}}}}\p{\dfrac{\p{1-\frac{k-2}{n-2}}\p{\frac{k-1}{n-1}}}{\P(\Lambda \not \subseteq S \mid \Lambda_1 \in S, A)}} \\
&= 3\ \frac{\p{1-\frac{k-2}{n-2}}\p{\frac{k-1}{n-1}}\p{\frac{k}{n}}}{1-\frac{A}{{n\choose2}}}.
\end{align*}

Thus, we have 
\[\P(\Lambda \subseteq S' \mid \Lambda \not \subseteq S, A) = 3 \p{\frac{\p{1-\frac{k-2}{n-2}}\p{\frac{k-1}{n-1}}\p{\frac{k}{n}}}{\p{1-\frac{A}{{n\choose2}}}}}\p{\frac{k-2}{{n\choose2}}} = 3\p{\frac{M}{{n\choose2} - A}},\] where $M = (n-k)\frac{{k\choose3}}{{n\choose3}}$. Hence,
\[\P(\Lambda \subseteq S' \mid A) = \p{1 - \frac{3(n-k)}{{n\choose2}}}\p{\frac{A}{{n\choose2}}} + 3\p{\frac{M}{{n\choose2}}},
\]
and so
\begin{align*}
\E[A' - A \mid A] &= -A + \sum_{\Lambda \subseteq \Z / n\Z} \P(\Lambda \subseteq S' \mid A)= -A + \p{1 - \frac{3(n-k)}{{n\choose2}}}\p{A} + 3\p{M} \\
&= -\p{\frac{3(n - k)}{{n\choose2}}}(A) + 3(n-k)\frac{{k\choose3}}{{n\choose3}} \\
&= -\frac{3(n-k)}{{n\choose2}}\p{A-\binom{n}{2}\frac{\binom{k}{3}}{\binom{n}{3}}} \\
&= -\frac{3(n-k)}{{n\choose2}}(A-\mu).
\end{align*}
Thus, 
$\E[A'-A \mid A] = -\lambda(A-\mu)$, where $\lambda = \frac{3(n-k)}{{n\choose2}}$. 
Hence, $\E[W'-W \mid W] = \E[\frac{A'-\mu}{\sigma} - \frac{A-\mu}{\sigma} \mid \frac{A-\mu}{\sigma}] = \frac{1}{\sigma}\E[A'-A \mid A] = -\lambda\p{\frac{A-\mu}{\sigma}} = -\lambda W$.

\end{proof}

Thus, we can bound the Wasserstein distance $\Wass(W, Z)$, where $Z$ is a standard Gaussian random variable using the method of exchangeable pairs. 
We have
\begin{align*}
\E\s{\frac{1}{2\lambda}(W'-W)^2 \mid W}
&= \E\s{\frac{{n\choose2}}{6(n-k)}(W'-W)^2 \mid W} \\
&= \frac{{n\choose2}}{6(n-k)}(\E[(W')^2 \mid W]-2\E[WW' \mid W]+W^2).
\end{align*}

\section{Attempt to prove nonexistence of LLT for $A_n$}

\subsection{General framework for nonexistence of an LLT for other R.V.'s}

We develop more general theorems for what properties a random variable $X_n$ can have to ensure it does not follow an LLT. Let $X_{n}$ (with mean $\mu_n$ and standard deviation $\sigma_n$) be conditioned on some event with value $k$, denoted by $X_{n,k}$. Let $\mu_{n,k} = \E[X_{n,k}]$ and $\sigma_{n,k}$ be the standard deviation of $X_{n,k}.$

The following theorem says that $X_n$ does not satisfy a LLT when the distances between the conditioned distributions $X_{n,k}$ are of a sufficiently larger order than their standard deviations $\sigma_{n,k}$. We show that at a point, the distribution of $X_n$ does not converge to the normal distribution. 

\begin{thm}
Suppose $\abs{x - \mu_{n,k}} = \omega(\sigma_{n,k} \sqrt{\sigma_n})$. Then
\[\abs{\P(X_n = x) - \mathcal{N}_n(x)} = \Omega\p{\frac{1}{\sigma_n}} \]
and there is no local limit theorem for $\{X_n\}$.
\end{thm}
\begin{proof}
Since we make no assumptions about the concentration of $X_{n,k}$, the best we can do is use Chebyshev's inequality.
\begin{align*}
\P(X_n = x)
&= \sum_{k=1}^j \P(X_n = x \mid Y_k) \P(Y_k) \\
&\leq \sum_{k=1}^j \P(\abs{X_{n,k} - \mu_{n,k}} \geq \abs{x - \mu_{n,k}}) \P(Y_k) \\
&\leq \sum_{k=1}^j \p{\f{\sigma_{n,k}}{\abs{x - \mu_{n,k}}}}^2 \P(Y_k) \\
&= \sum_{k=1}^j o\p{\f{1}{\sigma_n}} \P(Y_k) \\
&= o\p{\f{1}{\sigma_n}}.
\end{align*}
\end{proof}

In the previous theorem, the required distance between consecutive $X_{n,k}$ is rather large, and this is likely due to the inefficiency of using Chebyshev's inequality. If we assume a CLT on the $X_{n,k}$, we can use a better concentration inequality and require less of a distance between the $X_{n,k}$.

\begin{thm}
Suppose $X_{n,k}$ follows a CLT and $\abs{x - \mu_{n,k}} = \omega(\sigma_{n,k} \sigma_n^{\epsilon})$. Then
\[ \abs{\P(X_n = x) - \mathcal{N}_n(x)} = \Omega\p{\frac{1}{\sigma_n}} \]
and there is no local limit theorem for $\{X_n\}$.
\end{thm}
\begin{proof} We have
\begin{align*}
\P(X_n = x) &= \sum_{k = 1}^j \P(Y_k)\P(X_n = x \mid Y_k) \\
&\leq \sum_{k = 1}^j \P(\abs{X_{n, k}-\mu_{n,k}}\geq \abs{x-\mu_{n,k}}) \P(Y_k) \\
&\leq \sum_{k=1}^j \p{\frac{\sigma_{n,k}}{\abs{x-\mu_{n,k}}\sqrt{2\pi}}e^{-\p{\frac{x-\mu_{n,k}}{\sigma_{n,k}}}^2/2}} \P(Y_k) \\
&= \sum_{k=1}^j {o(\sigma_n^{-\epsilon})e^{-o(\sigma_n^{2\epsilon})} \P(Y_k)} \\
&= o(\sigma_n^{-\epsilon})o(\sigma_n^{-1+\epsilon}) = o(\sigma_n^{-1}).
\end{align*}

Thus, $\P(X_n = x)= o\p{\frac{1}{\sigma_n}}.$ Now $\mathcal{N}_n(x) = \frac{1}{\sqrt{2\pi}\sigma_n} e^{-\p{\frac{x-\mu_n}{\sigma_n}}^2/2} = \Theta\p{\frac{1}{\sigma_n}}$. Hence, \[\abs{\P(X_n = x) - \mathcal{N}_n(x)} = \Omega\p{\frac{1}{\sigma_n}}\] and thus $X_n$ does not follow a LLT. 

\end{proof}

\subsection{Variance of $A_{n,k}$}
\label{var-ank}

\begin{lem}
The variance of $A_{n,k}$ is $\sigma_{n,k}^2 = \frac{k^3(n-k)^3}{2n^4} + O(n)$.
\end{lem}

\begin{proof}
We have that $\sigma_{n,k}^2 = \E[A_{n,k}^2] - (\E[A_{n,k}])^2 = \E[A_{n,k}^2] - {n\choose 2}\frac{{k\choose 3}}{{n\choose 3}} = E[A_{n,k}^2] - \frac{3}{n-2} {k\choose 3}.$ We proceed to calculate $\E[A_{n,k}^2]$. Now $A_{n,k}^2 = (\sum_{\Lambda} 1_{\Lambda \subset S})^2. = \sum_{\Lambda_1}\sum_{\Lambda_2} 1_{\Lambda_1 \subset S} 1_{\Lambda_2 \subset S}.$ So 
\begin{align*}
\E[A_{n,k}^2] &= \sum_{\Lambda_1}\sum_{\Lambda_2} \E[1_{\Lambda_1 \subset S} 1_{\Lambda_2 \subset S}] \\
&= \sum_{i = 0}^3 \frac{{k\choose{6-i}}}{{n\choose{6-i}}} \sum_{\Lambda_1}\sum_{\substack{\Lambda_2 \textrm{ s.t. } \\ \abs{\Lambda_1 \cap \Lambda_2} = i}} 1.
\end{align*}

We now consider the quantity $\#\{\Lambda_2 \mid \abs{\Lambda_1 \cap \Lambda_2} = i\}$. Since we know there are $i$ elements pre-determined from $\Lambda_1$, we take each subset of $i$ elements from $\Lambda_1$ and perform the following procedure. Check how many arithmetic progressions contain those many numbers (which themselves come from an arithmetic progression), given the function $f: \bbn \times \{0, 1, 2, 3\} \to \bbn$ defined by $f(0) = \binom{n}{2}$, $f(1) = \frac{3}{2}(n-1)$, $f(2) = 3$, and $f(3) = 1$. In particular, we check $f(i)$. However, this overcounts the quantity we desire because it also counts arithmetic progressions whose intersection with $\Lambda_1$ is greater than 1. Thus, we check all the ways of adding another element from $\Lambda_1$ to our current subset of $i$ elements and subtract out $f(i+1)$ arithmetic progressions. We apply continue applying inclusion-exclusion until we arrive at the quantity we desire. Thus, we have

$$ \#\{\Lambda_2 \mid \abs{\Lambda_1 \cap \Lambda_2} = i\} = \binom{3}{i}\sum_{j=0}^{3-i} {(-1)^{j}\binom{i}{j}f(j-i)}. $$

Now, counting the number of pairs of 3-term arithmetic progressions $\{\Lambda_1, \Lambda_2\}$ that intersect 0, 1, 2, and 3 times, we have:
\begin{align*}
&\#\{\{\Lambda_1, \Lambda_2\} \mid \abs{\Lambda_1 \cap \Lambda_2} = 3\} = {n\choose 2} \\ 
&\#\{\{\Lambda_1, \Lambda_2\} \mid \abs{\Lambda_1 \cap \Lambda_2} = 2\} = 6{n\choose 2} \\
&\#\{\{\Lambda_1, \Lambda_2\} \mid \abs{\Lambda_1 \cap \Lambda_2} = 1\} = \frac{1}{2}{n\choose2}(9n-39) \\
&\#\{\{\Lambda_1, \Lambda_2\} \mid \abs{\Lambda_1 \cap \Lambda_2} = 0\} 
= \frac{1}{2}{n\choose 2}(n^2-10n+25).
\end{align*} 

After using these to calculate the second moment of $A_{n,k}$ by plugging the resulting expression into Mathematica, we find
\[ \sigma_{n,k}^2 = \frac{-(k-2)(k-1)k(k^3-3k^2(n-1) - n(n^2-3n+2)+k(3n^2-6n+2))}{2(n-4)(n-3)(n-2)^2}. \]
Simplifying, we have $\sigma_{n,k}^2 = \frac{k^3(n-k)^3}{2n^4} + O(n)$.
\end{proof}

\subsection{Nonexistence of LLT for $A_n$}
The reason that we cannot use the general theorem in section 6 to prove that no LLT exists for $A_n$ is that the distances between the conditioned distributions $A_{n,k}$ are the same order as their standard deviations $\sigma_{n,k}$, since $\mu_{n,k+1} - \mu_{n,k} = \Theta(\sigma_{n,k})$ for reasonably likely values of $k$ (when $k=1$, this breaks down but values of $k$ this extreme are very unlikely).

To show that there is no LLT for $A_n$, our strategy is to choose a point $x$ in the middle between two $A_{n,k}$ distributions and showing that $\abs{\P(A_n = x) - \mathcal{N}_n(x)} = \Omega(\frac{1}{\sigma_n})$, where $\mathcal{N}_n(x) = \frac{1}{\sqrt{2\pi}\sigma_n} e^{((x-\mu_n)/\sigma_n)^{2}/2}$. 

Set $p = \frac{1}{2}$, and let $x \geq \mu_{n,k}$ for all $k \leq j$ and $x \leq \mu_{n,k}$ for all $k>j$. Then we have
\begin{align*}
\P(A_n = x) &= \sum_{k = 0}^n \P(A_n = x \mid |S| = k)\P(|S| = k)\\
&= \sum_{k = 0}^n \P(A_{n,k} = x) \cdot \frac{{n\choose k}}{2^n} \\
&\leq \sum_{k = 0}^j \P(A_{n,k} \geq x) \cdot \frac{{n\choose k}}{2^n} + \sum_{k = j+1}^n \P(A_{n,k} \leq x) \cdot \frac{{n\choose k}}{2^n}.
\end{align*}

Now we must sufficiently bound the tails of $A_{n,k}$. We know that $\mu_{n} = \frac{1}{8}{n\choose 2}, \sigma_{n}^2 = \Theta(n^3), \mu_{n,k} = \frac{3}{n-2} {k\choose 3}, \text{ and } \sigma_{n,k}^2 = \frac{k^3(n-k)^3}{2n^4}$, but we do not know anything about the shape of the distribution of $A_{n,k}$. 

\begin{remark}
We conjecture that $A_{n,k}$ follows a local limit theorem. The main issue that we run into when trying to apply the same technique used to show LLTs for $T_n$ and $D_n$ to $A_{n,k}$ is that we cannot find an event to condition on such that $A_{n,k}$ will be the sum of independent random variables, since the constraint on size $k$ makes the presence of each element of $\Z/n\Z$ dependent on the presence of every other element. If we could prove an LLT for $A_{n,k}$, that would allow us to conclude that no LLT exists for $A_n$, as we show below. It might be possible to prove that there is no LLT for $A_n$ without proving the LLT for $A_{n,k}$ by a specialized argument, but using the following proposition would be the most satisfying way since an LLT for $A_{n,k}$ would give the most complete picture of the histogram of $A_n$.
\end{remark}

\begin{prop}
Suppose $A_{n,k}$ follows an LLT for each $k$. Then there exists no LLT for $A_n$.
\end{prop}

\begin{proof}
We show that the height at the center of an $A_{n,k}$ distribution is greater than the Gaussian distribution at that point as $n$ gets large. 
Let $k = \frac{n+1}{2}$ and $x = \mu_{n,k}$. Then since $A_{n,k}$ follows an LLT,
\begin{align*}
\P(A_{n} = x) &= \P(|S| = k) \cdot \P(A_{n,k} = x) = \frac{{n\choose k}}{2^n} \cdot \frac{1}{\sqrt{2\pi}\sigma_{n,k}} e^{-\p{\frac{x-\mu_{n,k}}{\sigma_{n,k}}}^2/2} \\
&= \p{\frac{1}{\sqrt{2\pi n}} e^{-\frac{(k-\frac{n}{2})^2}{2n}}} \cdot \p{\frac{1}{\sqrt{2\pi\sigma_{n,k}^2}}} \\
& \approx \frac{1}{2\pi} \sqrt{\frac{2n^4}{nk^3(n-k)^3}} \\
&= \frac{\sqrt{2}}{\pi} \p{\frac{n}{k(n-k)}}^{3/2} \\
&\approx \frac{8\sqrt{2}}{\pi} n^{-3/2}.
\end{align*}
Now we turn our attention to the height of the Gaussian distribution at $x = \mu_{n,k}.$ Recall that $\sigma_n^2 \approx \frac{9}{4}n^3$ and let $\mathcal{N}_n$ be the probability density function for the normal distribution with mean $\mu_n$ and variance $\sigma_n^2$. Then we have
\[ \mathcal{N}_n(x) = \frac{1}{\sqrt{2\pi}\sigma_{n}} e^{-\p{\frac{x-\mu_{n}}{\sigma_{n}}}^2/2} \leq \frac{1}{\sqrt{2\pi\sigma_{n}^2}} \approx \p{\frac{1}{3} \sqrt{\frac{2}{\pi}}} n^{-3/2}.
\]
Therefore,
\[\abs{\P(X_n = x) - \mathcal{N}_n(x)} \approx \p{\frac{8\sqrt{2}}{\pi} - \frac{1}{3} \sqrt{\frac{2}{\pi}}} n^{-3/2} = \Omega\p{\frac{1}{\sigma_n}},\]
so there exists no LLT for $A_n$.

\end{proof}



\section{Acknowledgements}
We would like to thank Felipe Hernandez for his close mentorship and for providing the topic of this project. We would also like to thank George Schaeffer for his helpful guidance and input on this project and for organizing SURIM, without which this project would not have been possible.

\section{Appendix}
\subsection{3-APs in a subset of $\Z/n\Z$ and its complement}
In this section, we prove a fun fact regarding 3-term APs.
	Let $n > 3$ be odd. Consider a subset $S \subseteq \Z/n\Z$. Define $A(S)$ to be the number of 3-term arithmetic progressions in $S$. That is,
	\[ A(S) = \frac{1}{2} \abs{\{a, b \in \Z/n\Z \mid b \neq 0;\ a, a + b, a + 2b \in S \}}. \]
	Let $k = \abs{S}$, and denote the complement of $S$ in $\Z/n\Z$ as $S^c$, so $\abs{S^c} = n - k$. We prove the following proposition.
    
    \begin{prop}
    The sum of the number of 3-APs in $S$ and $S^c$ is constant conditioning on $n$ and $k$. That is, $A(S) + A(S^c)$ depends only on $n$ and $k$.
    \end{prop}
    
    Note: this method of counting considers $(x, y, z)$ and $(z, y, x)$ to be the same arithmetic progression (due to the factor of 1/2), and does not consider $(x, x, x)$ to be an arithmetic progression. However, the result still holds under every permutation of these settings.

\begin{proof}

	The total number of 3-term arithmetic progressions $A(\Z/n\Z)$ is $n(n-1)$. These arithmetic progressions can be split into four groups based on the locations of their elements: all in $S$, all in $S^c$, exactly one in $S$, and exactly one in $S^c$.

	There are $A(S)$ progressions all in $S$ and $A(S^c)$ all in $S^c$. The number of progressions with exactly one element in $S$ is $\f{3}{2}k(k-1) - \f{3}{2}A(S)$ since each pair of elements of $S$ is contained in $3$ total progressions, but we need to subtract the contribution of the progressions entirely contained in $S$, each of which contains $3$ pairs. Similarly, the number of progressions with exactly one element in $S^c$ is $\f{3}{2}(n-k)(n-k-1) - \f{3}{2}A(S^c)$. Adding everything together, we get
	\[ n(n-1) = A(S) + \p{\f{3}{2}k(k-1) - \f{3}{2}A(S)} + \p{\f{3}{2}(n-k)(n-k-1) - \f{3}{2}A(S^c)} + A(S^c). \]
	Simplifying,
	\[ A(S) + A(S^c) = \f{1}{4}\p{3k(k-1) + 3(n-k)(n-k-1) - n(n-1)}. \]
	Surprisingly, this sum depends only on $n$ and $k$.
    
\end{proof}

Intuition tells us that $A(S)$ is a measure of structure in $S$ and the amount of structure in $S$ is the same as the amount of structure in $S^c$. But this intuition must be wrong since this calculation implies that either the amount of structure in $S^c$ is the inverse of the amount of structure in $S$, or $A(S)$ does not really measure structure. 

\bibliographystyle{plain}

\end{document}